\documentclass[]{amsart}
\usepackage{times,amsthm,amsmath, amssymb, color, mathrsfs, verbatim}

\newtheorem{thm}{Theorem}[section]

\newtheorem{lem}[thm]{Lemma}
\newtheorem{prop}[thm]{Proposition}

\numberwithin{equation}{section}

\begin{document}

\title[Commutants of Toeplitz operators]
{Commutants of Toeplitz operators with radial  symbols\\ on the Fock-Sobolev space}
\date{November 25, 2013}

\author[B. Choe]{Boo Rim Choe}
\address{Department of Mathematics, Korea University, Seoul 136-701, KOREA}
\email {cbr@korea.ac.kr}

\author[J. Yang]{Jongho Yang}
\address{Department of Mathematics, Korea University, Seoul 136-701, KOREA}
\email{cachya@korea.ac.kr}

\thanks{B. R. Choe was supported by Basic Science Research Program through the National Research Foundation of Korea(NRF) funded by the Ministry of Education, Science and Technology(2013R1A1A2004736)}

\begin{abstract}
In the setting of the Fock space over the complex plane,
Bauer and Lee have recently characterized commutants of Toeplitz operators with radial symbols,
under the assumption that symbols have at most polynomial growth at infinity. Their characterization states: If
one of the symbols of two commuting  Toeplitz operators is nonconstant and radial, then the other must be also radial.
We extend this result to the Fock-Sobolev spaces.
\end{abstract}

\maketitle

\section{Introduction}

Initiated by the seminal paper \cite{BH} of Brown and Halmos, the problem of characterizing commuting Toeplitz operators
has been one of the main topics in the study of Toeplitz operators on classical function spaces such
as the Hardy spaces and the Bergman spaces over various domains; see, for example, \cite{AC, CKL, CL, CR, GQV, L1, L2, V} and references therein.
While the problem with arbitrary symbols is still far from its solution, methods to handle the case of certain subclasses of symbols have been developed so far.
In particular, in the setting of the Bergman space over the unit disk, \u{C}u\u{c}ukovi\'c and  Rao \cite{CR} showed that if
one of the symbols of two commuting Toeplitz operators with bounded symbols is non-trivially radial, then the other also must be radial.
Recently Bauer and Lee \cite{BL} obtained an analogous result for the Fock space over the complex plane.
The purpose of the current paper is to extend the result of Bauer and Lee to the Fock-Sobolev spaces.

We first describe the function spaces to work on. For  $s\ge 0$ and an entire function $f$ with series expansion
$f(z)=\sum_{n=0}^{\infty} a_nz^n$, the $s$-order derivative $\mathcal R^s f$ is defined by
$$
\mathcal R^{s}f(z):=\frac{1}{(1+|z|)^{s}}\sum_{n=0}^{\infty} \frac{\Gamma(s+n+1)}{n!} a_nz^n.
$$
The {\em Fock-Sobolev space $F^{2,s}$ of order $s$} is then the space of all entire functions $f$ such that
the norm
$$
\|f\|_s := \|\mathcal R^s f \|_{L^2(GdA)}
$$
is finite. Here, $dA$ denotes the ordinary area measure on $\mathbb C$ and
$$
G(z): = \frac{1}{\pi}e^{-|z|^2}
$$
denotes the normalized Gaussian density.
The space $F^{2, s}$ turns out to be the closure in $L^2(G_s dA)$ of holomorphic polynomials where
$$
G_s(z) := |z|^{2s}G(z)
$$
denotes a weighted Gaussian density; see \cite[Theorem 4.2]{CCK}.
These Fock-Sobolev space was first studied in \cite{CZ} by means of ordinary derivatives in the case of positive integer orders and then generalized to arbitrary positive orders as above in \cite{CCK}. When $s=0$, note that the space $F^{2, 0}$ is simply the well-known Fock space over $\mathbb C$; we refer to a recent book \cite{Zhu} by Zhu for a systematic treatment of various aspects of the Fock space.

We now recall the notion of Toeplitz operators. Let
\begin{align}\label{projection}
P^s: L^2(G_sdA)\to F^{2, s}
\end{align}
be the Hilbert space orthogonal projection.
Also, let $\mathcal S_{\rm poly}$ be the class of all measurable functions on $\mathbb C$ having
at most polynomial growth at infinity. This means that a complex measurable function $u$ on $\mathbb C$ belongs to $\mathcal S_{\rm poly}$ if and only if
there is a constant $C>0$ and an exponent $m>0$ such that
\begin{align}
\label{symbol}
|u(z)|\le C(1+|z|)^m
\end{align}
for almost every $z\in \mathbb C$. Of course, the term ``almost every" here refers to the measure $dA$.
Now, for $u\in \mathcal S_{\rm poly}$, the {\em Toeplitz operator $T^s_u$ with symbol $u$} denotes the densely-defined linear operator on $F^{2,s}$ given by
\begin{align}
\label{toeplitz}
T^s_u f = P^s(uf)
\end{align}
for holomorphic polynomials $f$.
In Section \ref{prod}
we will see that products of finitely many Toeplitz operators with symbols in $\mathcal S$ are still densely-defined on $F^{2, s}$.

For $u,v \in \mathcal S_{\rm poly}$, we denote by
$$
[T_u^s, T_v^s ] : = T_u^s T_v^s - T_v^s T_u^s
$$
the commutator of $T_u^s$ and $T_v^s$.
We write $[T_u^s, T_v^s ]=0$  if $[T_u^s, T_v^s ]$ annihilates all holomorphic polynomials.
For example, it is not hard to check $[T_u^s, T_v^s ]=0$ for radial symbols $u$ and $v$. In fact
a Toeplitz operator with radial symbol is easily seen to be a diagonal operator with respect to the orthonormal basis formed by the normalized monomials.
The following theorem is our main result.

\begin{thm}\label{main}
Let $s\ge 0$ and $u,v \in \mathcal S_{\rm poly}$. Assume that $u$ is nonconstant and radial.
If $[T_u^s, T_v^s ]=0$, then $v$ is also radial.
\end{thm}

Our proof will be based on the main scheme of \cite{BL} with 
certain amount of extra work required
for the setting of the Fock-Sobolev spaces.

\section{Preliminaries}\label{pre}

In this section we collect some basic facts which we need in the proof Theorem \ref{main}.
The parameter $s\ge 0$ is fixed throughout the discussion in this section.

\subsection{Reproducing kernels}
It is known (\cite[Proposition 2.2]{CCK}) that there is a constant $C=C(s)>0$ such that
$$
|f(z)|\le C\frac{ e^{\frac{|z|^2}{2}} }{(1+|z|)^s} \|f\|_s,\qquad z\in \mathbb C
$$
for $f\in F^{2,s}$. This implies that each point evaluation is continuous on $F^{2, s}$. So, to each $z\in\mathbb C$, there corresponds
a unique reproducing kernel $K^s_z\in F^{2, s}$ such that
$$
f(z) = \langle f, K^s_z \rangle_s,\qquad f\in F^{2,s}
$$
where $\langle \ ,   \rangle_s$ denotes the inner product of $L^2(G_sdA)$.
So, the orthogonal projection \eqref{projection} can be realized as an integral operator
$$
P^{s}\psi(z)=\langle \psi, K^s_z \rangle_s,\qquad z\in \mathbb C
$$
for $\psi\in L^2(G_{s}dA)$.
The explicit formula for $K^{s}(z,w): =\overline{K^s_z(w)}$ is given by
\begin{align}
\label{kernel}
K^{s}(z,w)=\sum_{n=0}^\infty \frac{(z\overline w)^n }{\Gamma(s+n+1)};
\end{align}
see \cite[Theorm 4.5]{CCK}.

For the pointwise growth rate of $K^{s}(z,w)$, it is known that, given $0<a<1$, there are positive constants $C=C(a, s)$ and $\delta=\delta(a)$ such that
\begin{align}
\label{growth}
|K^{s}(z,w)| \le C \frac{ e^{{\rm Re} (z \overline w)} \chi_\delta(z\overline w) + e^{a |z||w|}}{(1+|z||w|)^s},\qquad z, w\in\mathbb C
\end{align}
where $\chi_\delta$ is the characteristic function of the angular sector consisting of all nonzero complex numbers
$\lambda$ such that $|\arg \lambda|<\delta$; see \cite[Corollary 4.6]{CCK}.
Also, for the norm estimate, we have
\begin{align}
\label{norm}
\|K^{s}_z\|_s = \sqrt{K^s(z, z)}\le C \frac{ e^{\frac{|z|^2}{2}}} {(1+|z|)^s},\qquad z\in\mathbb C
\end{align}
for some constant $C=C(s)>0$; see \cite[Proposition 4.8]{CCK}.

\subsection{Products of Toeplitz operators}\label{prod}
Since we are to consider commutators of densely-defined Toeplitz operators on $F^{2,s}$, we need to verify carefully that products of such operators are still densely defined  on $F^{2,s}$.

We first introduce a symbol class which contains $\mathcal S_{\rm poly}$.
Fix $s\ge 0$.
Given $\epsilon\ge 0$, we denote by $\mathcal D^s_\epsilon$ the Banach space of all complex measurable functions $u$ on $\mathbb C$ equipped with the norm
$$
\|u\|_{\mathcal D^s_\epsilon}:= \left\| u(z)(1+|z|)^s e^{-\epsilon|z|^2} \right\|_{L^\infty(dA)}.
$$
We put
\begin{align}\label{sym}
\mathcal S^s:= \bigcap_{\epsilon>0} \mathcal D^{s}_{\epsilon}\quad \text{and}\quad \mathcal D^s: = \bigcup_{0<\epsilon<1/2} \mathcal D^{s}_{\epsilon}.
\end{align}
For example, functions having a linear exponential growth at infinity are all contained in $\mathcal S^s$.
In particular, we have $\mathcal S_{\rm poly}\subset \mathcal S^s$. Also, note
$\mathcal D^s\subset  L^2(G_s dA)$.

In the next lemma $\Lambda^s$ denotes the integral operator  defined by
$$
\Lambda^s \psi(z) = \int_{\mathbb C} \psi(z)|K^s(z,w)| G_s(w)\,dA(w)
$$
for $\psi$ that makes the integral well defined. Note $|P^s\psi|\le \Lambda^s|\psi|$.

\begin{lem}
\label{shift}
Given $s\ge 0$ and $0\le \epsilon<1/2$, there is a constant $C=C(s,\epsilon)>0$ such that
$$
|\Lambda^s u(z)| \le C e^{\frac{|z|^2}{4(1-\epsilon)}} \|u\|_{\mathcal D^s_\epsilon},\qquad z\in\mathbb C
$$
for $u\in\mathcal D^s_\epsilon$.
\end{lem}

\begin{proof}
Fix $s\ge 0$ and $0\le \epsilon<1/2$. Note from \eqref{growth} that
\begin{align}
\label{growth2}
|K^s(z,w)|\le C_1 \frac{e^{{\rm Re}(z\overline w)}+e^{\frac12|z||w|}}{(1+|z||w|)^s}
\end{align}
for some constant $C_1 =C_1(s)>0$.

Let $u\in\mathcal D^s_\epsilon$. It is easily seen from \eqref{growth2} that $|\Lambda^s u(z)|$, $|z|\le 1$,  stays bounded
by $\|u\|_{\mathcal D^{s}_\epsilon}$ times some constant (depending on $\epsilon$ and $s$).
So, assume $|z|\ge 1$. Since $1+|z||w|\ge 1+|w|$, we have again by \eqref{growth2}
\begin{align*}
|\Lambda^{s}u(z)|
&\leq \frac{C_1}{\pi} \|u\|_{\mathcal D^{s}_\epsilon}[I_1(z) + I_2(z)]
\end{align*}
where
\begin{align*}
I_1(z) = \int_{\mathbb C} e^{-(1-\epsilon)|w|^2+{\rm Re}(z\overline w)} \,dA(w)\\
\intertext{and}
I_2(z) = \int_{\mathbb C} e^{-(1-\epsilon)|w|^2+\frac12|z||w|} \,dA(w).
\end{align*}
For the first integral, we have
\begin{align*}
I_1(z)
&= e^{\frac{|z|^2}{4(1-\epsilon)}} \int_{\mathbb C} e^{-(1-\epsilon) \left| w- \frac{z}{2(1-\epsilon)}\right|^2} \, dA(w)\\
&= e^{\frac{|z|^2}{4(1-\epsilon)}} \int_{\mathbb C} e^{-(1-\epsilon) |w|^2} \, dA(w).
\end{align*}
For the second integral, we have
\begin{align*}
\frac{I_2(z)}{2\pi}
&=  \int_0^{\infty} e^{-(1-\epsilon)r^2+\frac12|z|r}r\,dr\\
&=  e^{\frac{|z|^2}{16(1-\epsilon)}}\int_0^{\infty} e^{-(1-\epsilon)\left[r-\frac{|z|}{4(1-\epsilon)}\right]^2}r\,dr\\
&\le e^{\frac{|z|^2}{16(1-\epsilon)}} \int_{-\infty}^{\infty} e^{-(1-\epsilon)t^2}\left |t+\frac{|z|}{4(1-\epsilon)}\right |\,dt\\
&\le 2 e^{\frac{|z|^2}{16(1-\epsilon)}}\int_0^\infty e^{-\frac{1}{2}t^2}(t+|z|)\,dt;
\end{align*}
the second inequality comes from $0\le \epsilon<1/2$. Note that the last integral is dominated by some absolute constant times $(1+|z|)$.
Combining these estimates, we conclude the lemma.
\end{proof}

Based on Lemma \ref{shift}, we introduce a sequence of positive numbers $\{\epsilon_j\}$ defined inductively by
$\epsilon_1=\frac{1}{4}$ and $\epsilon_{j+1}= \frac{1}{4(1-\epsilon_j)}$ for integers $j\ge 1$. 
More explicitly, we have
\begin{align}\label{seq}
\epsilon_j = \frac12 - \frac{1}{2j+2}
\end{align}
for each $j$.  Also, put
\begin{align*}
\mathcal F^s_j:= \mathcal D^s_{\epsilon_j} \cap F^{2, s}
\end{align*}
for each $j$ and
\begin{align*}
\mathcal F^s_\infty:& = \bigcup_{j=1}^\infty \mathcal F^{s}_j.
\end{align*}
Note that $\mathcal F^s_\infty$ contains all holomorphic polynomials and thus is densely contained in $F^{2,s}$.
Now, consider a scale of Banach spaces:
\begin{align}\label{chain}
\mathbb C \subset \mathcal F^{s}_1 \subset \mathcal F^{s}_2 \subset \cdots \subset \mathcal F^{s}_\infty  \subset F^{2,s};
\end{align}
each space $\mathcal F^s_j$ is considered as a closed subspace of $\mathcal D^s_{\epsilon_j}$.
Given an integer $k\ge 0$, denote by $\mathscr L_k(\mathcal F^s_\infty)$ the class of all linear operators $\Lambda$ on $\mathcal F^s_\infty$
such that
$$
\Lambda: \mathcal F^s_j\to \mathcal F^s_{j+k}
$$
is bounded for each $j\ge 1$. So, each operator in $\mathscr L_k(\mathcal F^s_\infty)$ might be considered as a ``$k$-order shift"
with respect to the  scale \eqref{chain}. Also, we put
$$
\mathscr L_{\rm fos}(\mathcal F^{s}_\infty):=\bigcup_{k=0}^\infty \mathscr  L_k(\mathcal F^s_\infty )
$$
where the term ``{\rm fos}" stands for ``finite-order shift". Clearly, $\mathscr L_{\rm fos}(\mathcal F^s_\infty)$ is an algebra.

Note
$P^s \in \mathscr L_2(\mathcal F^s_\infty)$ by Lemma \eqref{shift}.
Also, for $u\in \mathcal S^s$, denoting by $M_u$ the pointwise multiplication $\psi\mapsto u\psi$, one easily verifies
 $P^sM_u\in \mathscr L_3(\mathcal F^s_\infty)$.
So, extending symbols of Toeplitz operator in \eqref{toeplitz} to functions in $\mathcal S^s$, we have the following proposition.

\begin{prop}
\label{welldefined}
Given $s\ge 0$, the product of any finitely many Toeplitz operators with symbols in $\mathcal S^s$ is densely defined on $F^{2,s}$.
\end{prop}

\subsection{Berezin transform}
The Berezin transform has been one of main tools in the operator theory  on various Hilbert spaces of holomorphic functions.
For example, the vanishing Berezin transform usually indicates that the operator under consideration must be the zero operator.
Such one-to-one property of the Berezin transform remains valid on the Fock spaces; see \cite[Lemma 12]{B}. Here, we check that it is still
valid on the Fock-Sobolev spaces.

For an operator $A\in \mathscr L_{\rm fos}(\mathcal F^{s}_\infty)$, its {\em $s$-Berezin transform}
$\mathcal B^s [A]$ is as usual a continuous (in fact real-analytic) function on $\mathbb C$ defined by
$$
\mathcal B^{s}[A](z):=\frac{\langle AK_z^{s}, K_z^{s}\rangle_{s}}{\|K_z^s\|_s^2}
$$
for $z\in\mathbb C$. Note that integral in the right-hand side of the above is well defined, because
$K_z^s\in\mathcal F^s_\infty$ by \eqref{growth}.

\begin{lem}\label{1-1}
For any $s\ge 0$, the $s$-Berezin transform is one-to-one on $\mathscr L_{\rm fos}(\mathcal F^{s}_\infty)$.
\end{lem}

When $s=0$, the lemma is proved in \cite[Lemma 12]{B}. The proof below is similar and included for completeness.

\begin{proof}
Let $s\ge 0$. We first introduce an auxiliary class of function spaces. Given $0<r<1$, put $G_{s,r}(z):=\frac{1}{\pi}|z|^{2s}e^{-r|z|^2}$ and
$$
\mathcal H^{s,r}:= H(\mathbb C) \cap L^2(G_{s,r}dA)
$$
where $H(\mathbb C)$ denote the set of all entire functions on $\mathbb C$.
Of course, the space $\mathcal H^{s,r}$ is regarded as a Hilbert subspace of $L^2(G_{s,r}dA)$.
Since $0<r<1$, the embedding
$$
\mathcal H^{s,r}\hookrightarrow F^{2,s}
$$
is clearly continuous.
By \cite[Lemma 2.1]{CCK} there is a constant $C=C(r,s)>0$ such that
\begin{align}\label{growth3}
|f(z)| \leq C \|f\|_{L^2(G_{s,r}dA)}\frac{e^{\frac{r}{2}|z|^2}}{(1+|z|)^{s}}
\end{align}
for $f\in  \mathcal H^{s,r}$ and $z\in\mathbb C$.
We also note from \eqref{kernel} that with respect to the orthonormal basis of $\mathcal H^{s,r}$ formed by the normalized monomials
$$
e_{r,n}(w):=\sqrt{\frac{r^{n+s+1}}{\Gamma(n+s+1)}}~w^n,\qquad n=0, 1, 2, \dots,
$$
the kernel $K_z^s$ can be expanded as
\begin{align}\label{L2conv}
K^{s}_z=\sum_{n=0}^\infty \frac{\overline z^n}{\sqrt{\Gamma(n+s+1)r^{n+s+1}}}e_{r,n},
\end{align}
which is easily seen to be convergent in $\mathcal H^{s,r}$.

Now, using the sequence $\{\epsilon_j\}$ defined in \eqref{seq}, pick a sequence $\{r_j\}$ of positive numbers such that
$$
\epsilon_j<\frac{r_j}{2}<\epsilon_{j+1}, \qquad j= 1, 2, 3, \dots
$$
and put $\mathcal H^s_{j}:=\mathcal H^{s, r_j}$ for each $j$.
Obviously, the embeddings
$$
\mathcal F^s_j\hookrightarrow \mathcal H^s_{j}
$$
are all continuous. Also, the embeddings
$$
\mathcal H^s_j\hookrightarrow \mathcal F^s_{j+1}
$$
are all continuous by \eqref{growth3}.
Thus we have a scale of Banach spaces:
\begin{align}\label{scale}
\mathbb C  \subset \mathcal F^{s}_1  \subset \mathcal H^s_1  \subset  \mathcal F^{s}_2 \subset \mathcal H^s_2 \subset
\cdots \subset \mathcal F^{s}_\infty=\mathcal H^{s}_\infty
\end{align}
where $\mathcal H^{s}_\infty:= \cup_{j=1}^\infty \mathcal H^s_j$.

Let $A\in\mathscr L_{\rm fos}(\mathcal F^s_\infty)$. We see from \eqref{scale} that
there is a positive integer $k$ such that $A : \mathcal H^{s}_j \to \mathcal H^{s}_{j+k}$ is continuous for each $j$. In particular, we have by
\eqref{L2conv} (with $r=r_1$)
$$
AK^s_{z}=\sum_{n=0}^\infty\frac{\overline z^n }{\sqrt{\Gamma(n+s+1)r_1^{n+s+1}}}~A[e_{r_1,n}]\in \mathcal H^s_{1+k}\subset F^{2,s}.
$$
This shows that the function $z\mapsto \overline{AK^s_{z}}$ is an $F^{2,s}$-valued entire function.
Therefore the function
$$
f_A(z,w):=\langle K^s_{\overline w}, AK^s_z\rangle_s= \overline {AK^s_z(\overline w)}
$$
is an entire function of two variables $z$ and $w$; the second equality comes from the reproducing property.

Now, further assume that the $s$-Berezin transform of $A$ is identically zero.
Then, since $f_A(z,\overline z)=0$ for all $z\in \mathbb C$, we see that $f_A$ is identically zero on $\mathbb C \times \mathbb C$; see, for example, \cite[Exercise 3, p. 371]{K}.
In other words, $A$ annihilates all the reproducing kernels $K^s_z$.
Note that the span of $\{K^s_z: z\in \mathbb C\}$ is dense in $F^{2,s}$ and thus in $\mathcal H^s_j$ for each $j$.
Accordingly, we conclude that $A$ is the zero operator on $\mathcal H^{s}_\infty=\mathcal F^{s}_\infty$.
The proof is complete.
\end{proof}

\subsection{ $L^2$-decomposition of symbols}

In what follows $z=re^{i\theta}$ denotes the polar expression of $z\in\mathbb C$.

\begin{lem}\label{decompose}
Each $v \in \mathcal S_{\rm poly}$ admits an $L^2(G_sdA)$-convergent expansion of the form
$$
v(z)=\sum_{j=-\infty}^{\infty} v_j(r)e^{ij\theta}, \qquad z=re^{i\theta}
$$
where each $v_j \in \mathcal S_{\rm poly}$ is a radial function on $\mathbb C$.
\end{lem}

\begin{proof}
Let $v\in \mathcal S_{\rm poly}$ and put $v_r(\lambda)= v(r\lambda)$ for $r>0$.
Since
\begin{align}\label{identity}
\|v\|^2_{L^2(G_sdA)}
= \dfrac{1}{\pi \Gamma(s+1)}\int_0^{\infty} \left\{\int_0^{2\pi}|v_r(e^{i\theta})|^2~d\theta\right\}~ r^{2s+1}e^{-r^2}~dr < \infty,
\end{align}
we see that $v_r \in L^2(\mathbb T)$, the usual Lebesgue space over the unit circle $\mathbb T$ with respect to the arc-length measure, for a.e. $r>0$. For such $r$, we have an $L^2(\mathbb T)$-convergent Fourier-series:
$$
v(r e^{i\theta})=\sum_{j=-\infty}^\infty v_j(r)e^{ij\theta}
$$
where
$$
v_j(r):=\frac{1}{2\pi}\int_0^{2\pi} v(re^{i\theta})e^{-ij\theta}~d\theta
$$
denotes the $j$-th (measurable) Fourier coefficient of $v_r$.

Each $v_j$, when extended to a radial function on $\mathbb C$, is easily seen to belong to
$\mathcal S_{\rm poly}$, because $v\in\mathcal S_{\rm poly}$.
Moreover, we have by \eqref{identity} and Parseval's identity
$$
\sum_{j=-\infty}^\infty\int_0^{\infty}|v_j(r)|^2~r^{2s+1}e^{-r^2}\, dr=\pi\Gamma(s+1)\|v\|^2_{L^2(G_sdA)}<\infty.
$$
Accordingly, given $\epsilon$, one can find a finite set $J$ of integers such that
\begin{align*}
\epsilon
&>\frac{2}{\Gamma(s+1)}\sum_{j\notin J}\int_0^{\infty}|v_j(r)|^2~r^{2s+1}e^{-r^2}\, dr\\
&= \int_{\mathbb C} \Biggl| \sum_{j\notin J}v_j(r)e^{ij\theta} \Biggr|^2 G_s(z)\, dA(z)\qquad (z=re^{i\theta}).
\end{align*}
This shows the $L^2(G_sdA)$-convergence of the series in question. The proof is complete.
\end{proof}

\section{Proofs}\label{proof}

As in \cite{BL}, we first observe that commuting property of two Toeplitz operators
implies a collection of functional equations involving the Mellin transforms associated with symbol functions and the density function.

So, before proceeding, we first recall the well-known notion of the Mellin transform.
Let $\mathbb R^+$ be the set of all positive numbers.
Given a locally integrable function $f$ on $\mathbb R^+$, its {\em Mellin transform} $\mathcal M[f]$
is defined by
$$
\mathcal M[f](z):=\int_0^{\infty} f(x)x^{z-1}\,dx
$$
at $z\in\mathbb C$ for which the integral exists. So, $\mathcal M[f]$ is defined at $z$ if $|f(x)|x^{{\rm Re}(z)-1}\in L^1(\mathbb R^+)$
and, for example,  this constraint is met in the vertical strip $a<{\rm Re}(z)<b$
when $f(x)x^a= \mathcal O(1)$ as $x\to 0^+$ and  $f(x)x^{b}= \mathcal O(1)$ as $x\to \infty$.

We also recall some results about the Mellin transform proved in \cite{BL}.
Denote by $\mathcal A$ the class of all complex measurable functions $u$ on
$\mathbb R^+$ such that
\begin{align}\label{classA}
 \left|u\left(\frac1x\right)x^{-\rho}\right|=\mathcal O(1)\quad \text{and} \quad  |u(x)x^{-\eta}|=\mathcal O(1)\quad  \text{for}\quad  x \ge 1
\end{align}
for some $\rho, \eta\ge 0$. The next proposition is taken from \cite[Propositions 4.11 and 4.16]{BL}.
In conjunction with the second statement of the next proposition, note
\begin{align}\label{mellinG}
\mathcal M[G](2z)=\dfrac{1}{2\pi}\Gamma(z)
\end{align}
for ${\rm Re}(z)>0$.

\begin{prop}[\cite{BL}]\label{lemBL}
Given $u\in\mathcal A$, the following statements hold:
\begin{enumerate}
\item[(a)] If $0<a\le 2$ and if
$$
\int_0^{\infty} u(t)e^{-t}t^{ak}\, dt=0
$$
for all large positive integers $k$, then $u=0$ a.e. on $\mathbb R^+$;
\item[(b)] If the function $$
z\mapsto \frac{\mathcal M[uG](2z+2)}{\Gamma(z+1)}
$$ is
a periodic entire function whose period is a positive integer, then $u$ is constant a.e. on $\mathbb R^+$.
\end{enumerate}
\end{prop}

Now, we derive a collection of functional equations
by justifying that our density function satisfies certain conditions given in \cite{BL}. For nonnegative integers $n$, set
$$
a_n:= \int_{\mathbb C} |z|^{2n}G_{s}(z)\, dA(z)=\Gamma(s+n+1).
$$
We then have by \eqref{kernel} and \eqref{norm}
$$
\sum_{n=0}^\infty \frac{|z|^n}{a_n} \le C \frac{e^{\frac{|z|}{2}}}{(1+\sqrt{|z|})^s},\qquad z\in \mathbb C
$$
for some constant $C=C(s)>0$. Thus we have
$$
u(z) |z|^m \sum_{n=0}^\infty \frac{|z|^n}{a_n}\in L^1 (G_s dA)
$$
for any $u\in \mathcal S_{\rm poly}$ and integer $m\ge 1$. Note $a_n^{1/n}\to \infty$ as $n\to \infty$ by Sterling's formula.
This shows that the density function $G_s$ satisfies the conditions required in \cite[Proposition 2.4]{BL}. Also, recall that
holomorphic polynomials form a dense subset of $F^{2,s}$. So, the following is a consequence of \cite[Proposition 2.4]{BL}.
In what follows we identify a radial function on $\mathbb C$ with its  its restriction to $\mathbb R^+$.

\begin{lem}
Let $s\ge 0$.
Let $u,v \in \mathcal S_{\rm poly}$ and assume that $u$ is nonconstant and radial.
If $[T^s_u, T^s_v]=0$, then
\begin{align}\label{ftleq}
\left\{\dfrac{\mathcal M[uG_s](2k+2)}{\mathcal M[G_s](2k+2)}-\dfrac{\mathcal
M[uG_s](2k+2j+2)}{\mathcal M[G_s](2k+2j+2)}\right\} \mathcal M[v_jG_s](j+2k+2)=0
\end{align}
for all integers $k\geq0$ and $j$ with $j+k\geq 0$. Here, $v_j$ denotes the function provided by the decomposition of $v$ in Lemma \ref{decompose}.
\end{lem}

In conjunction with \eqref{ftleq} we note
\begin{align*}
\mathcal M[uG_s](z)
&=\frac{1}{\pi}\int_0^{\infty} u(t)e^{-t^2}t^{z+2s-1}\,dt
\end{align*}
is holomorphic on the half-plane ${\rm Re}(z)>-2s$. Also, note that $\mathcal M[uG_s]$ and $\mathcal M[uG]$ are related by
\begin{align*}
\mathcal M[uG_s](z)=\mathcal M[uG](2s+z).
\end{align*}
In particular, we have by \eqref{mellinG}
\begin{align*}
\mathcal M[G_s](2z)=\dfrac{1}{2\pi}\Gamma\left(s+z\right)
\end{align*}
for ${\rm Re}(z)>-s$. Accordingly,  \eqref{ftleq} can be rephrased as
\begin{align}\label{mainequ}
\Phi_j(k+s)\Psi_j(k+s)=0
\end{align}
where
\begin{align*}
\Phi_j(z)&=\dfrac{\mathcal M[uG](2z+2)}{\Gamma(z+1)}-\dfrac{\mathcal M[uG](2z+2j+2)}{\Gamma(z+j+1)}
\intertext{and}
\Psi_j(z)&= \mathcal M[v_jG](j+2z+2).
\end{align*}
Since the Gamma function is non-vanishing, we see that
each $\Phi_j$ is holomorphic on the half-plane ${\rm Re}(z)>\max\{-1, -j-1\}$.

In what follows we use the notation
$$
Q_{s,j}(z):= \Phi_j(z+s)\Psi_j(z+s)\Gamma(z+s+1),\qquad {\rm Re}(z)>-s-1
$$
for $s\ge 0$ and positive integers $j$.
Recall that $\mathcal A$ denotes the class of functions on $\mathbb R^+$ satisfying\eqref{classA}.

\begin{lem}
\label{inverse}
Given $s\ge 0$ and a positive integer $j$, there is some $f\in \mathcal A$ such that
\begin{align}
\label{qsj}
Q_{s,j}(z)=\mathcal M[f(x)e^{-x}](2z)
\end{align}
for ${\rm Re}(z)>-s-1$.
\end{lem}

\begin{proof}
In case $s=0$, the lemma is proved in \cite[Proposition 4.10]{BL}. So, there is some $g\in\mathcal A$ such that
$$
Q_{j,0}(z)=\mathcal M[g(x)e^{-x}](2z)
$$
for ${\rm Re}(z)>-1$. Now, one may  check that the lemma holds with the function $f(x):= g(x)x^{2s}$.
\end{proof}

Note $({T^s_u})^\ast =T^s_{\bar u}$ for $u\in L^\infty(dA)$. Thus, for $u, v\in L^\infty(dA)$, we always have
$[T_u^s, T^s_v]=0$ if and only if $[T^s_{\bar u}, T^s_{\bar v}]=0$.
We need to extend this property to symbols under consideration. Recall that $\mathcal S^s$ denotes the symbol class defined in \eqref{sym}.

\begin{lem}\label{adjoint}
Let $u,v \in \mathcal S^s$. If $[T^s_u, T^s_v]=0$, then $[T^s_{\bar u}, T^s_{\bar v}]=0$.
\end{lem}

\begin{proof}
For the Berezin transform of $T^s_u T^s_v$ at an arbitrary point $z\in\mathbb C$, we have
\begin{align*}
\mathcal B^s[T^s_u T^s_v](z)
&=\langle u T^s_v(k^s_z), k^s_z \rangle_s\\
&=\langle T^s_v(k^s_z), \overline u k^s_z \rangle_s\\
&=\langle T^s_v(k^s_z), T^s_{\overline u} (k^s_z) \rangle_s\\
&=\langle v k^s_z, T^s_{\overline u} (k^s_z) \rangle_s\\
&=\langle k^s_z, \overline v T^s_{\overline u} (k^s_z) \rangle_s\\
&=\overline{\mathcal B^s[T^s_{\bar v}T^s_{\bar u}](z)};
\end{align*}
the equalities above are easily verified, because all the functions inside the inner products belong to $\mathcal D^s\subset L^2(G_s dA)$.
Thus we have $\mathcal B^s[T^s_u T^s_v] = \overline{\mathcal B^s[T^s_{\bar v}T^s_{\bar u}]}$.
By interchanging the roles of $u$ and $v$, we also obtain
$\mathcal B^s[T^s_v T^s_u]=\overline{\mathcal B^s [T^s_{\bar u} T^s_{\bar v}]}$.
So, assuming $[T^s_u, T^s_v]=0$, we see that the $s$-Berezin transform of $[T^s_{\bar u}, T^s_{\bar v}]$
is identically zero. Thus we conclude $[T^s_{\bar u}, T^s_{\bar v}]=0$ by Lemma [\ref{1-1}].
\end{proof}

Now, we are ready to conclude the proof of Theorem \ref{main}.

\begin{proof}[Proof of Theorem \ref{main}]
Assume $[T_u^s, T_v^s]=0$. To derive a contradiction suppose that $v$ is not radial.
Then there is some integer
$j\ne 0$ such that $v_j$ is nontrivial. We may assume that such $j$ is positive, because $[T_{\overline u}^s, T_{\overline v}^s]=0$
by Lemma \ref{adjoint}. Put
\begin{align*}
\Phi_{s,j}(z)= \Phi_j(z+s), \quad\Psi_{s,j}(z)= \Psi_j(z+s)\\
\intertext{so that}
Q_{s,j}(z):=\Phi_{s,j}(z)\Psi_{s,j}(z)\Gamma(z+s+1)
\end{align*}
for ${\rm Re}(z)>-s-1$.

Using Lemma \ref{inverse}, pick a function $f\in\mathcal A$ such that \eqref{qsj} holds. We then have by \eqref{mainequ}
\begin{align*}
Q_{s,j}(k)= \int_0^{\infty}f(x)e^{-x}x^{2k-1}\,dx=0
\end{align*}
for any integer $k\ge 0$. So, we see by Proposition \ref{lemBL}(a) (with $a=2$) that $f=0$ a.e. on $\mathbb R^+$
 and hence that $Q_{s,j}$ vanishes everywhere on the half-plane ${\rm Re}(z)>-s-1$.

Note that the holomorphic function $\Psi_{s,j}$ is not identically zero on a right half-plane, because the Mellin transform is one-to-one.
Since the gamma function is non-vanishing, we see that $\Phi_{s,j}$ vanishes everywhere on the half-plane ${\rm Re}(z)>-s-1$. So, setting
$$
H(z):=\dfrac{\mathcal M[uG](2z+2)}{\Gamma (z+1)},
$$
we see that
$$
H(z+s)=H(z+s+j),\qquad {\rm Re}(z)>-s-1,
$$
or said differently,
$$
H(z)=H(z+j),\qquad {\rm Re}(z)>-1.
$$
This shows that the function $H$ extends to an entire function with period $j$. Finally, we conclude
by Proposition \ref{lemBL}(b) that $u$ is constant a.e. on $\mathbb R^+$, which is a contradiction. The proof is complete.
\end{proof}

\end{document}